\documentclass{article}
\usepackage[utf8]{inputenc}
\usepackage{amsmath}
\usepackage{amsfonts}
\usepackage{amssymb,dsfont}
\usepackage{amsthm}
\usepackage{comment}
\usepackage{diagbox}

\newtheorem{theorem}{Theorem}[section]
\newtheorem{corollary}{Corollary}[theorem]
\newtheorem{lemma}[theorem]{Lemma}
\theoremstyle{definition}

\theoremstyle{remark}

\theoremstyle{definition}

\newcommand{\F}{\mathbb{F}}

\title{On the max min of the algebraic degree and the nonlinearity of a Boolean function on an affine subspace}
\author{Jan Kristian Haugland}

\begin{document}

\maketitle

\begin{abstract}
We investigate the max min of the algebraic degree and the nonlinearity of a Boolean function in $n$ variables when restricted to a $k$-dimensional affine subspace of $\F_2^n$. Previous authors have focused on the cases when the max min of the algebraic degree is 0 or 1. Upper bounds, lower bounds and a conjecture on the exact value in special cases are presented.
\end{abstract}

\section{Introduction}
Let $f$ be a Boolean function in $n$ variables. With $k \leq n$, let $\alpha(f, k)$ denote the minimal algebraic degree of $f$ when restricted to a $k$-dimensional affine subspace of $\F_2^n$ (subsequently referred to as a $k$-dimensional \textit{flat}) and let $g(n, k) = \max_f \alpha(f, k)$.

Similarly, let $\alpha'(f, k)$ denote the minimal nonlinearity of $f$ when restricted to a $k$-dimensional flat, and let $g'(n, k) = \max_f \alpha'(f, k)$. The nonlinearity of a Boolean function is the minimal Hamming distance to an affine function.

$f$ is said to be $k$-normal if $\alpha(f, k) = 0$ and weakly $k$-normal if $\alpha(f, k) \leq 1$ (which is equivalent to $\alpha'(f, k) = 0$), and these cases have been of particular interest. Dubuc \cite{dubuc} proved that for $n \leq 7$, any Boolean function in $n$ variables is $\lfloor \frac{n}{2} \rfloor$-normal, i.e., $g(n, k) = 0$ if $n \leq 7$ and $k = \lfloor \frac{n}{2} \rfloor$. In this note, we give an upper bound (Theorem 2.5) and a lower bound (Theorem 2.3) for $g(n, k)$, as well as considering exact values when $n-k \in \{1, 2\}$ (Theorem 2.2 and the conjecture in Section 5) and in some single cases (Section 3).

We focus mostly on $g(n, k)$, but we also give a lower bound for $g'(n, k)$ and values in some single cases.

\section{General bounds}
\begin{lemma}
The algebraic degree of $f$ when restricted to a $k$-dimensional flat ($k \geq 1$) is equal to $k$ if and only if the sum of $f(x)$ over all vertices of the flat is 1.
\end{lemma}

\begin{proof}
Confer \cite{cusick}, Theorem 2.5.
\end{proof}

\begin{theorem}
If $n \geq 2$, then $g(n, n-1) = n-2$.
\end{theorem}

\begin{proof}
First, we prove that $g(n, n - 1) < n - 1$. This follows almost directly from Lemma 2.1. We only need to identify an $(n - 1)$-dimensional flat on which the sum of $f(x)$ is 0.

Let $S$ be any $(n - 2)$-dimensional flat. $S$ must have three additional cosets which are also flats. There must be two of them on which $f$ has the same sum, and so we can simply combine them into an $(n - 1)$-dimensional flat on which $\sum f(x)=0$ as required. 

Second, we verify that there exists an example of a function $f$ for which $\alpha(f, n - 1) \geq n - 2$. We can use the function that maps $x$ to 1 if the weight of $x$ is 0 or 1, and 0 otherwise. A flat of dimension $n-1$ must contain at least one $x$ such that $f(x)=1$, otherwise the complement would contain all of them and generate the whole of $\F_2^n$. Then the intersection with either $x_i = 1$ (if $x$ is 1 in the $i$th coordinate) or $\sum x_i=0$ (if $x$ is (0, 0, ..., 0)) contains exactly one vertex $y$ for which $f(y) = 1$, and $f$ must have maximal degree $\geq n-2$ here. Therefore, $f$ is also of degree $\geq n-2$ on the original flat.
\end{proof}

\begin{theorem}
If $n, k, d$ are integers with $n > k > d \geq 0$ satisfying $$(k+1)(n-k) + 2 \leq \binom{k}{0} + \dotsc + \binom{k}{k-d-1}$$ then $g(n, k) > d$. Similarly, if $m < 2^{k-2}$ and $$\binom{2^k}{0} + \dotsc + \binom{2^k}{m} \leq 2^{2^k - (k+1)(n-k+1) -2}$$ then $g'(n, k) > m$.
\end{theorem}

\begin{proof}
We begin by observing that $\left( 1-\frac{1}{2^i} \right)^{2^{i-1}} \geq \frac{1}{2}$ for $i \geq 1$, and thus $\prod_{i \geq 1} \left( 1-\frac{1}{2^i} \right) \geq \left(\frac{1}{2}\right)^{1+\frac{1}{2}+\frac{1}{4}+\dotsc} = \frac{1}{4}$. (We could have used the exact value 0.288788... for the infinite product, but $\geq \frac{1}{4}$ is good enough here.) We can now estimate the number of $k$-dimensional flats in $\F_2^n$, which is given by $$2^{n-k} \frac{(2^n-1) (2^{n-1}-1) \dotsc (2^{n-k+1}-1)}{(2^k-1) (2^{k-1}-1) \dotsc (2^1-1)}$$ We have $$(2^n-1) (2^{n-1}-1) \dotsc (2^{n-k+1}-1) < 2^{k(n-\frac{k-1}{2})}$$ $$(2^k-1) (2^{k-1}-1) \dotsc (2^1-1) > \left( \prod_{i \geq 1} \left( 1-\frac{1}{2^i} \right) \right) 2^{k\frac{k+1}{2}} \geq \frac{1}{4} \times 2^{k\frac{k+1}{2}}$$ $$\implies 2^{n-k} \frac{(2^n-1) (2^{n-1}-1) \dotsc (2^{n-k+1}-1)}{(2^k-1) (2^{k-1}-1) \dotsc (2^1-1)} < 2^{(n-k) + k(n-k) + 2}$$ $$= 2^{(k+1)(n-k)+2}$$ For each of these flats, the proportion of functions of algebraic degree $\leq d$ on it is $2^{-\left(\binom{k}{0} + \dotsc + \binom{k}{k-d-1}\right)}$. By the condition of the Theorem, the proportion of functions of algebraic degree $\leq d$ on at least one flat must be smaller than 1, and it follows that there must exist a function $f:\F_2^n \to \F_2$ such that $\alpha(f, k) > d$.

The argument only uses the scarceness of functions of low degree, and a similar argument applies to nonlinearity (or any other property, for that matter). Since $m < 2^{k-2}$, the proportion of functions of nonlinearity $\leq m$ on any given $k$-dimensional flat is $$\left( \binom{2^k}{0} + \dotsc + \binom{2^k}{m} \right) 2^{k+1-2^k}$$ and the bound for $g'(n, k)$ follows.
\end{proof}

\begin{corollary}
Suppose $n, k$ are integers and that $k \geq 5$. If $k + 2 \leq n \leq \frac{3k-1}{2}$, then $g(n, k) \geq k-2$. If $\frac{3k}{2} \leq n \leq \frac{(k+1)(k+4)}{6}$, then $g(n, k) \geq k-3$.
\end{corollary}

\begin{proof}
Regarding the first claim, we can apply Theorem 2.3 with $d=k-3$. Since $n \leq \frac{3k-1}{2}$, we have $$(k+1)(n-k)+2 \leq \frac{k^2+3}{2} < \frac{k^2+k+2}{2} = \binom{k}{0} + \binom{k}{1} + \binom{k}{2}$$ For the second claim, we put $d=k-4$ instead. Since $n \leq \frac{(k+1)(k+4)}{6}$, we have $$(k+1)(n-k)+2 \leq \frac{k^3+3k+16}{6} \leq \frac{k^3+5k+6}{6} = \binom{k}{0} + \binom{k}{1} + \binom{k}{2} + \binom{k}{3}$$ In both cases, the condition of Theorem 2.3 is satisfied.
\end{proof}

We can also give an upper bound by simple means.

\begin{lemma}
If $a$, $b$ and $c$ are nonnegative integers satisfying $$\max (a + 1, b) \leq c \leq a+b$$ then $$\lceil 2^a \frac{2^b - 1}{2^c - 1} \rceil = 2^{a+b-c}$$
\end{lemma}

\begin{proof}
We have $$2^a \frac{2^b - 1}{2^c - 1} = 2^a \left( \frac{1}{2^{c-b}} - \frac{1-2^{b-c}}{2^c - 1} \right) > 2^a \left( \frac{1}{2^{c-b}} - \frac{1}{2^c - 1} \right)$$ and since $a < c$, it follows that $$2^a \frac{2^b - 1}{2^c - 1} > 2^{a+b-c} - 1$$ On the other hand, $b \leq c$ implies that $$2^a \frac{2^b - 1}{2^c - 1} \leq 2^{a+b-c}$$ Lastly, since $c \leq a + b$, it is clear that $2^{a+b-c}$ is an integer, and the conclusion follows.
\end{proof}

\begin{theorem}
If $k > d \geq 0$ and $n \geq 2^{k-1} + k - \lfloor 2^{d-1} \rfloor$, then $g(n, k) \leq d$.
\end{theorem}

\begin{proof}
Two $\delta$-dimensional flats $\{x_i\}$ and $\{y_j\}$ are said to be \textit{parallel} if there exists a vector $z \in \F_2^n$ such that $\{x_i + z\}$ = $\{y_j\}$. Note that there are exactly $2^{n-\delta}$ flats that are parallel to $\{x_i\}$ (including itself), all mutually disjoint.

First, we prove the case $d = 0$. It suffices to assume that $n = 2^{k-1} + k$ and verify that $\alpha(f, k) = 0$ for any function $f$ in $n$ variables.

For $0 \leq \delta \leq k$, let $h(\delta)$ denote the maximal number of parallel $\delta$-dimensional flats on which $f$ is constant. Trivially, $h(0) = 2^n$.

For $\delta = 1, 2, ...$, there are at least $\lceil \frac{h(\delta - 1)(h(\delta - 1) - 2)}{4} \rceil$ pairs of disjoint parallel $(\delta - 1)$-dimensional flats on which $f$ is equal to the same constant, while there are $2^{n-\delta+1}-1$ ways to pair up all $(\delta-1)$-dimensional flats parallel to a given one into parallel $\delta$-dimensional flats. Thus we have $$h(\delta) \geq \lceil \frac{h(\delta-1)(h(\delta-1)-2)}{4(2^{n-\delta+1}-1)} \rceil$$ We can now show that $h(\delta) \geq 2^{n-(2^\delta+\delta)+1}$ for $0 \leq \delta \leq k-1$ by induction on $\delta$. By Lemma 2.4 with $a = b = n-(2^{\delta-1}+\delta)+1$, $c = n - \delta + 1$ we have the induction step $$\lceil \frac{h(\delta-1)}{2} \frac{\frac{h(\delta-1)}{2}-1}{2^{n-\delta+1}-1} \rceil = 2^{n-(2^\delta+\delta)+1}$$ In particular, $h(k-1) \geq 4$, so we have at least two different parallel $(k-1)$-dimensional flats on which $f$ is the same constant, and we can combine them into a $k$-dimensional flat. This completes the case $d=0$. (It is straightforward to verify that the conditions on $a$, $b$ and $c$ are satisfied.)

Second, when $d > 0$, the idea is to not distinguish between the admissible $(\delta-1)$-dimensional flats when forming $\delta$-dimensional flats for $\delta > k - d$. Then $f$ will have algebraic degree at most 1 on the $(k-d+1)$-dimensional flats, at most 2 on the $(k-d+2)$-dimensional flats, and so on.

With $n = 2^{k-1} + k - 2^{d-1}$, the proof that $h(k-d) \geq 2^{n-(2^{k-d}+k-d)+1}$ is similar as above, and we omit the details. Then, for $k-d+1 \leq \delta \leq k$, let $h'(\delta)$ denote the maximal number of parallel $\delta$-dimensional flats formed by the admissible $(\delta-1)$-dimensional flats, and put $h'(k-d) = h(k-d)$. We have $$h'(\delta) \geq \lceil \frac{h(\delta-1)(h(\delta-1)-1)}{2(2^{n-\delta+1}-1)} \rceil$$ and as above, we can show that $h'(\delta) \geq 2^{n-(2^\delta+\delta)+2^{\delta-k+d}}$ for $k-d+1 \leq \delta \leq k-1$. By Lemma 2.4 with $a = n - (2^{\delta-1} + \delta) + 2^{\delta - k + d - 1}$, $b = a+1$, $c = n - \delta + 1$, we have the induction step $$\lceil \frac{h'(\delta-1)}{2} \frac{h'(\delta-1)-1}{2^{n-\delta+1}-1} \rceil =  2^{n-(2^\delta+\delta)+2^{\delta-k+d}}$$ Thus we have $h'(k-1) \geq 2$, i.e., two different parallel $(k-1)$-dimensional flats that we can combine into a $k$-dimensional flat on which $f$ has algebraic degree at most $d$.

\end{proof}

\section{Single cases}
We have verified that $g(5, 3) = 1$, $g(6, 4) = 2$ and $g(7, 5) = 3$ with the help of Langevin's classification of Boolean functions in up to six variables \cite{langevin1} and in seven variables \cite{langevin2}. (A few more details are given in Section 5.) Polujan \textit{et al.} \cite{polujan} found a bent function in 8 variables that is not 4-normal, and Canteaut \textit{et al.} \cite{canteaut} found a bent function in 10 variables that is not weakly 5-normal. In other words, $g(8, 4) \geq 1$ and $g(10, 5) \geq 2$.

Given positive integers $k$ and $d$, an algorithm for searching for a function $f$ such that $\alpha(f, k) \geq d$ is as follows. Start with a random function $f_0$ in $k$ variables. For each $i$, $1 \leq i \leq N$, let $f_i$ be given by the same truth table as that of $f_{i-1}$ except at one or two randomly chosen points. If the number of flats on which the degree is $<d$ is higher for $f_i$ than for $f_{i-1}$, discard the new $f_i$ and let $f_i \leftarrow f_{i-1}$ instead. If the number of such flats has dropped to zero within $N$ steps, we are done. If not, we may be stuck in a local dead end and can try again from the top. In this manner we were able to find a function $f^{(7, 4)}:\F_2^7 \to \F_2$ such that $\alpha(f^{(7, 4)}, 4) = 2$ and a function $f^{(8, 5)}:\F_2^8 \to \F_2$ such that $\alpha(f^{(8, 5)}, 5) = 3$. Langevin found equivalent functions with fewer terms in the algebraic normal form (private communication), specifically \newline \newline $f^{(7, 4)}(x_1, \dotsc, x_7) \sim x_1x_2x_3x_4x_5x_7 \oplus x_1x_2x_3x_4x_5 \oplus x_1x_2x_3x_4x_6x_7 \oplus x_1x_2x_3x_4 \oplus x_1x_2x_4x_5x_6 \oplus x_1x_2x_4x_5x_7 \oplus x_1x_2x_6x_7 \oplus x_1x_2x_6 \oplus x_1x_2x_7 \oplus x_1x_3x_4x_5x_6 \oplus x_1x_3x_4x_5 \oplus x_1x_3x_4x_7 \oplus x_1x_3x_5x_7 \oplus x_1x_4x_5x_7 \oplus x_1x_4x_5 \oplus x_1x_4x_6x_7 \oplus x_1x_4x_6 \oplus x_1x_6x_7 \oplus x_1x_7 \oplus x_2x_3x_4x_5 \oplus x_2x_3x_5x_6x_7 \oplus x_2x_3x_5x_7 \oplus x_2x_4x_6x_7 \oplus x_2x_4x_6 \oplus x_2x_5x_6x_7 \oplus x_3x_4x_5x_6 \oplus x_3x_4x_5x_7 \oplus x_3x_4x_6 \oplus x_3x_4 \oplus x_3x_5 \oplus x_3x_7 \oplus x_4x_5x_7 \oplus x_5x_7 \oplus x_6x_7$ \newline \newline and \newline \newline $f^{(8, 5)}(x_1, \dotsc, x_8) \sim x_1x_2x_3x_4x_5x_6x_8 \oplus x_1x_2x_3x_4x_5x_6 \oplus x_1x_2x_3x_4x_6x_7x_8 \oplus$ \newline $x_1x_2x_3x_4x_6x_7 \oplus x_1x_2x_3x_4x_6 \oplus x_1x_2x_3x_5x_6x_7 \oplus x_1x_2x_3x_5x_8 \oplus x_1x_2x_3x_6 \oplus x_1x_2x_4x_5x_6x_8 \oplus x_1x_2x_4x_5x_6 \oplus x_1x_2x_4x_6 \oplus x_1x_2x_4x_7 \oplus x_1x_2x_5x_6x_7 \oplus x_1x_2x_6x_7 \oplus x_1x_2x_6 \oplus x_1x_3x_4x_5x_8 \oplus x_1x_3x_4x_5 \oplus x_1x_3x_4x_6 \oplus x_1x_3x_4x_7x_8 \oplus x_1x_3x_4x_7 \oplus x_1x_3x_5x_7x_8 \oplus x_1x_3x_5x_7 \oplus x_1x_3x_5x_8 \oplus x_1x_3x_6x_7x_8 \oplus x_1x_3x_8 \oplus x_1x_4x_5x_6x_7x_8 \oplus x_1x_4x_5x_6x_7 \oplus x_1x_4x_6 \oplus x_1x_4x_7 \oplus x_1x_4x_8 \oplus x_1x_5x_7 \oplus x_1x_7x_8 \oplus x_2x_3x_4x_5x_6x_7x_8 \oplus x_2x_3x_4x_5x_7x_8 \oplus x_2x_3x_4x_5x_8 \oplus x_2x_3x_4x_6x_7x_8 \oplus x_2x_3x_4x_6x_8 \oplus x_2x_3x_4x_7 \oplus x_2x_3x_4x_8 \oplus x_2x_3x_5x_6x_7x_8 \oplus x_2x_3x_5x_6x_7 \oplus x_2x_3x_5 \oplus x_2x_3x_6x_7x_8 \oplus x_2x_3x_6x_8 \oplus x_2x_3x_7 \oplus x_2x_4x_5x_6x_7x_8 \oplus x_2x_4x_5x_6x_7 \oplus x_2x_4x_5x_7x_8 \oplus x_2x_4x_6x_7x_8 \oplus x_2x_4x_6x_8 \oplus x_2x_4x_6 \oplus x_2x_4x_7x_8 \oplus x_2x_5x_6x_7 \oplus x_2x_5x_7x_8 \oplus x_2x_5x_8 \oplus x_2x_6x_7x_8 \oplus x_3x_4x_5x_6x_7x_8 \oplus x_3x_4x_5x_6x_7 \oplus x_3x_4x_5x_8 \oplus x_3x_4x_6x_8 \oplus x_3x_4x_6 \oplus x_3x_4x_7 \oplus x_3x_5x_6x_7 \oplus x_3x_5x_6x_8 \oplus x_3x_5x_7x_8 \oplus x_3x_5x_8 \oplus x_3x_6x_7x_8 \oplus x_3x_6x_7 \oplus x_4x_5x_6x_7x_8 \oplus x_4x_5x_6x_8 \oplus x_5x_7x_8 \oplus x_6x_7x_8$ \newline \newline Thus, since $g(7, 4) \leq g(6, 4)$ and $g(8, 5) \leq g(7, 5)$, we have $g(7, 4) = 2$ and $g(8, 5) = 3$.

As for nonlinearity, the optimal functions are of course bent or almost bent when $k=n$, which gives $g'(2, 2) = 1$, $g'(3, 3) = 2$, $g'(4, 4) = 6$, $g'(5, 5) = 12$ and $g'(6, 6) = 28$. Using Langevin's classifications, we found that this is common also in other cases with small $n$ and $k$ and nonzero value of $g'(n, k)$. We have $g'(4, 3) = 2$, $g'(5, 4) = 4$, $g'(6, 4) = 4$ and $g'(6, 5) = 12$.

\section{Tabulated values}
Taking all the above results and single cases into account, we get the following values and lower bounds for $g(n, k)$ when $n \leq 12$ and $k \leq 6$ (upper bounds, where not specified, can be inferred from lower values for $n$).
\newline

\begin{tabular}{|c|c|c|c|c|c|c|c|c|c|c|c|c|}\hline
\diagbox{$k$}{$n$} & 1 & 2 & 3 & 4 & 5 & 6 & 7 & 8 & 9 & 10 & 11 & 12 \\ \hline
1 & 1 & 0 & 0 & 0 & 0 & 0 & 0 & 0 & 0 & 0 & 0 & 0 \\
2 & & 2 & 1 & 0 & 0 & 0 & 0 & 0 & 0 & 0 & 0 & 0 \\
3 & & & 3 & 2 & 1 & 0 & 0 & 0 & 0 & 0 & 0 & 0 \\
4 & & & & 4 & 3 & 2 & 2 & $\geq 1$ & $\geq 0$ & $\geq 0$ & 0 or 1 & 0 \\
5 & & & & & 5 & 4 & 3 & 3 & $\geq 2$ & $\geq 2$ & $\geq 0$ & $\geq 0$ \\
6 & & & & & & 6 & 5 & $\geq 4$ & $\geq 3$ & $\geq 3$ & $\geq 3$ & $\geq 2$\\ \hline
\end{tabular}
\newline
\newline

Here are the corresponding values and bounds for $g'(n, k)$. As previously indicated, $g(n, k) \leq 1 \iff g'(n, k) = 0$ and $g(n, k) \geq 2 \iff g'(n, k) \geq 1$.
\newline

\begin{tabular}{|c|c|c|c|c|c|c|c|c|c|c|c|c|}\hline
\diagbox{$k$}{$n$} & 1 & 2 & 3 & 4 & 5 & 6 & 7 & 8 & 9 & 10 & 11 & 12 \\ \hline
1 & 0 & 0 & 0 & 0 & 0 & 0 & 0 & 0 & 0 & 0 & 0 & 0 \\
2 & & 1 & 0 & 0 & 0 & 0 & 0 & 0 & 0 & 0 & 0 & 0 \\
3 & & & 2 & 2 & 0 & 0 & 0 & 0 & 0 & 0 & 0 & 0 \\
4 & & & & 6 & 4 & 4 & $\geq 1$ & $\geq 0$ & $\geq 0$ & $\geq 0$ & 0 & 0 \\
5 & & & & & 12 & 12 & $\geq 3$ & $\geq 2$ & $\geq 1$ & $\geq 1$ & $\geq 0$ & $\geq 0$ \\
6 & & & & & & 28 & $\geq 16$ & $\geq 12$ & $\geq 9$ & $\geq 7$ & $\geq 5$ & $\geq 3$\\ \hline
\end{tabular}
\newpage
\section{A conjecture}
We think the following statement is true.\newline

\noindent\textbf{Conjecture:} Let $k$ be an integer $\geq 2$. Then $g(k + 2, k) = k - 2$.
\newline

For $k=2$, the claim follows from the result of Dubuc, and for $k\in\{3, 4, 5\}$, this has also been verified, as mentioned in Section 3.

For $k \geq 6$, Corollary 2.3.1 says that $k-2$ is indeed a valid lower bound for $g(k + 2, k)$. On the other hand, we can give a heuristic argument for why $g(k + 2, k)$ is probably not as high as $k-1$. The algebraic degree of the restriction of a random function $f\colon\F_2^{k+2} \to \F_2$ to a given $k$-dimensional flat is $\geq k-1$ with probability $1-\frac{1}{2^{\binom{k}{0} + \binom{k}{1}}} = 1-\frac{1}{2^{k+1}}$, and so the probability that this holds for all $k$-dimensional flats should be approximately $$(1-\frac{1}{2^{k+1}})^{2^2 \frac{(2^{k+2}-1)(2^{k+1}-1)}{3}}$$ (of course, the probabilities are not actually independent, which is why this is a heuristic argument and not a proof). For $k$ sufficiently large, this is approximately $e^{-\frac{4}{3}2^{k+2}}$. There are "only" $2^{2^{k+2}}$ Boolean functions in $k+2$ dimensions, which suggests that the expected number of functions for which every restriction to a $k$-dimensional flat is of degree $\geq k-1$ should be about $(\frac{2}{e^{4/3}})^{2^{k+2}}$ which tends to zero quite rapidly. Actually, we only need to check one representative from each of a much smaller number of classes of functions, which diminishes the chances of the existence of such a function even further.

For $k \in \{2, 3, 4, 5\}$, a computer search for the Boolean functions in $k+2$ variables given in \cite{langevin1}, \cite{langevin2} of algebraic degree $\leq k-2$ on the \textit{smallest} number of $k$-dimensional flats gave the following results.
\newline

\noindent\begin{tabular}{|c|c|c|}\hline
$k$ & Function & No. of flats \\ \hline
2 & $x_1 x_2 x_3 \oplus x_1 x_4 \oplus x_2$ & 10 / 140 \\ \hline
3 & $x_1 x_2 x_3 x_4 \oplus x_1 x_2 x_5 \oplus x_1 x_4 \oplus x_2 x_3$ & 15 / 620 \\ \hline
4 & $x_1 x_2 x_3 x_4 \oplus x_1 x_3 x_5 \oplus x_1 x_4 x_6$ & 20 / 2604 \\
& $\oplus x_2 x_3 x_5 \oplus x_2 x_3 x_6 \oplus x_2 x_4 x_5$ &  \\ \hline
5 & $x_1 x_2 x_3 x_4 x_5 x_6 \oplus x_1 x_2 x_3 x_4 x_7 \oplus x_1 x_2 x_4 x_5 \oplus x_1 x_2 x_6 x_7$ & 73 / 10668\\
& $\oplus x_1 x_3 x_4 x_5 \oplus x_1 x_3 x_4 x_6 \oplus x_1 x_3 x_5 x_7 \oplus x_2 x_3 x_5 x_6$ & \\ \hline
\end{tabular}
\newline

\noindent \textsc{Norwegian National Security Authority (NSM), Norway}

\noindent \textit{E-mail address:} \texttt{admin@neutreeko.net}

\end{document}